\documentclass[12pt]{au}
\usepackage{amsmath,amsthm,amssymb}

\newtheorem{thm}{Theorem}[section]

\newtheorem{lm}[thm]{Lemma}

\newtheorem{prb}[thm]{Problem}

\newcommand{\m}[1]{{\mathbf{\uppercase{#1}}}}
\newcommand{\vr}[1]{{\mathcal{\uppercase {#1}}}}
\newcommand{\ppf}{{\rm PPF}}
\newcommand{\eqr}{{\rm Eq}}
\newcommand{\fo}{{\rm FO3}}
\newcommand{\ra}{{\rm RA}}
\newcommand{\con}{{\rm Con}}

\def\id{{1\textrm{'}}}

\newcommand{\1}{1^\text{'}}
\newcommand{\ds}{\displaystyle}
\begin{document}

\title[Equivalence Relations and Relation Algebras]{Lattices of Equivalence Relations Closed Under the Operations of Relation Algebras}
\author{Jeremy F. Alm \and John W. Snow}
\date{\today}
\address{}
\email{}

\subjclass{}
\keywords{}
\bibliographystyle{plain}

\begin{abstract}
One of the longstanding problems in universal algebra is the question of which finite lattices are isomorphic to the congruence lattices of finite algebras.  This question can be phrased as which finite lattices can be represented as lattices of equivalence relations on finite sets closed under certain first order formulas.  We generalize this question to a different collection of first-order formulas, giving examples to demonstrate that our new question is distinct.  We then prove that every lattice $\m M_n$ can be represented in this new way. [This is an extended version of a paper submitted to \emph{Algebra Universalis}.]
\end{abstract}

\maketitle

\section{Introduction}

One of the longstanding problems in universal algebra is,

\begin{prb} \label{FCLRP} {\bf Finite Congruence Lattice Representation Problem:}
For which finite lattices $\m l$ is there a  finite algebra $\m a$ with $\m l \cong \con \m a$?
\end{prb}

A {\it primitive positive formula} is a first-order formula of the form $\exists \wedge({\rm atomic})$.  Suppose that $\vr R$ is a set of relations on a finite set $A$.  Let $\ppf(\vr r)$ be the set of all relations on $A$ definable using primitive positive formulas and relations from $\vr R$.  Let $\eqr(\vr R)$ be the set of all equivalence relations in $\vr R$. It follows from  \cite{bod,PoschelKaluznin1979} that $\vr R$ is the set of all universes of direct powers of an algebra $\m a$ with universe $A$ if and only if $\ppf(\vr R)=\vr R$.  (For references on similar characterizations,  the reader can see \cite{PosRel}.)  Therefore, Problem \ref{FCLRP} can be restated in the following way.

\begin{prb} \label{PPFP}
For which finite lattices $\m l$ is there a lattice $\vr l$ of equivalence relations on a finite set so that $\m l \cong \vr l$ and $\vr l=  \eqr(\ppf(\vr L))$?
\end{prb}

A natural extension of this problem is to consider first-order definitions employing types of formulas other than primitive positive formulas.  We suggest replacing primitive positive formulas here with any first-order formulas using at most three variables.  If $\vr R$ is a set of relations on a finite set $A$, let $\fo(\vr R)$ be the set of all relations on $A$ definable using first-order formulas with at most three variables and relations from $\vr R$.  Our extension of \ref{PPFP} can be stated as:

\begin{prb} \label{FO3P}
For which finite lattices $\m l$ is there a lattice $\vr l$ of equivalence relations on a finite set so that $\m l \cong \vr l$ and $\vr l = \eqr(\fo(\vr L))$?
\end{prb}

Our interest in first-order formulas with three variables stems from
a connection with relation algebras.  A {\it relation algebra} is an
algebra $\mathbf{A}=\langle A,+,\bar{\cdot},;,\cdot^{\cup},\1\rangle$ with  operations intended mimic the
operations of union, complement, composition,  converse, and identity
on binary relations.  A relation algebra $\m a$ is {\it
representable} if there is a set of binary relations $\vr R$ on a
set $B$ so that $\m a$ is isomorphic to the algebra $\langle \vr R,
\cup, \bar{\cdot}, \circ, \cdot^{\cup}, \1_B\rangle$.  A set $\vr R$ of binary
relations on a finite set $A$ is closed under the relation algebra
operations if and only if every binary relation on $A$ definable
with a first-order formula with at most three variables and
relations in $\vr R$ is already in $\vr R$ (see Theorem 3.32 of
\cite{games} or page 172 of \cite{thebook}).  For any set $\vr R$ of
binary relations on a set $A$, let $\ra (\vr R)$ be the relation
algebra generated by $\vr R$.  Then the above problem becomes:

\begin{prb} \label{RAP}
For which finite lattices $\m l$ is there a lattice $\vr l$ of equivalence relations on a finite set so that $\m l \cong \vr l$ and $\vr l =  \eqr(\ra(\vr L))$?
\end{prb}

For any relation algebra $\m a$, let $\eqr(\m a)$ be equivalence relation elements of $\m a$.  Then our problem becomes:

\begin{prb}
For which finite lattices $\m l$ is there a relation algebra $\m a$ which is representable on a finite set so that $\m l \cong \eqr(\m a)$?
\end{prb}
\section{Examples}

In this section we give two examples $\vr L$ and $\vr M$ of lattices of equivalence relations on finite sets.
In the first example, $\eqr(\ppf(\vr L)) = \vr L$ but $\eqr(\ra(\vr L)) \neq \vr L$.
In the second example,  $\eqr(\ra(\vr M)) = \vr M$ but $\eqr(\ppf(\vr M)) \neq \vr M$.
This demonstrates that these two notions are indeed distinct.

First, let $\m 2$ be the two-element lattice with universe $\{0,1\}$.  Let $\m
a = \m 2^2$, and let $\vr L=\con \m a$.  Then $\vr L$ contains four
equivalence relations -- the identity relation, the universal relation,
and the kernels of the projection homomorphisms.  The projection kernels
are the relations $\eta_0$ and $\eta_1$ defined so that $(x_0,x_1) \eta_0
(y_0,y_1)$ when $x_0=y_0$ and  $(x_0,x_1) \eta_1
(y_0,y_1)$ when $x_1=y_1$.  Since $\vr L$ is a
congruence lattice, $\eqr(\ppf(\vr L)) = \vr L$.
However, $\ra(\vr L)$ also contains the equivalence relation
$$\gamma = \id \cup \overline{(\eta_0 \cup \eta_1)}$$ which is not in $\vr L$,
so   $\eqr(\ra(\vr L)) \neq \vr L$.
Note that the relation $\gamma$ can also be defined with this  first-order formula which only uses {\it two} variables:
$$x \gamma y \leftrightarrow (x=y) \vee \neg[(x\eta_0 y) \vee(x \eta_1 y)].$$
Thus $\vr L$ is closed under
primitive positive definitions but not under the operations of relation
algebras or first-order definitions using at most three variables.

For our second example, suppose that $p\geq5$ is prime.
We consider $\con(\mathbb{Z}^2_p)$, which is a copy of $\m M_{p+1}$ consisting of
the identity $\id$, the universal relation $1$, and  $p+1$
atoms $\eta_0,\eta_1,\alpha_1,\ldots,\alpha_{p-1}$, given by
\begin{align*}
  \langle x_0,x_1\rangle\eta_0\langle y_0,y_1\rangle &\leftrightarrow x_0=y_0\\
  \langle x_0,x_1\rangle\eta_1\langle y_0,y_1\rangle &\leftrightarrow x_1=y_1\\
  \langle x_0,x_1\rangle\alpha_1 \langle y_0,y_1\rangle &\leftrightarrow 1 x_0-x_1=1 y_0-y_1\\
  \langle x_0,x_1\rangle\alpha_2\langle y_0,y_1\rangle &\leftrightarrow
  2x_0-x_1=2y_0-y_1 \\
  &\vdots\\
  \langle x_0,x_1\rangle\alpha_k\langle y_0,y_1\rangle &\leftrightarrow
  kx_0-x_1=ky_0-y_1\\
&\vdots\\
  \langle x_0,x_1\rangle\alpha_{p-1}\langle y_0,y_1\rangle &\leftrightarrow (p-1)x_0-x_1=(p-1)y_0-y_1
\end{align*}

\begin{lm}\label{lemma1}
 Suppose that $1\leq n <p-2$ and let $\vr M=\{1, \id, \eta_0, \eta_1, \alpha_1, \ldots, \alpha_n\}$.
Then   $\eqr(\ra(\vr M)) = \vr M$.
\end{lm}

This lemma follows from \cite{Lyndon};  the result is not explicitly stated in the paper, although it can be extracted from it. A proof is given in Section \ref{mainsection}.

  Consider the relation $\alpha_{p-1}$ (which is not in $\vr M$); $\alpha_{p-1}$  can be defined from $\eta_0$, $\eta_1$, and $\alpha_1$ with  a primitive positive formula by
$$a \alpha_{p-1} b \leftrightarrow \exists c, d \left( a \eta_0 c \wedge c \eta_1 b \wedge a \eta_1 d \wedge d \eta_0 b \wedge c \alpha_1 d
\right).$$
Thus $\eqr(\ppf(\vr M)) \neq \vr M$.  The lattice $\vr m$ is closed under
the operations of relation algebras and first-order definitions
using at most three variables but not under primitive positive
definitions.

This second example has the following interesting consequence.  If $n\geq 1$ and if $p\geq5$ is a prime greater than $n+2$, then the lattice $\vr M$ in the example gives a lattice of equivalence relations closed under the operations of relation algebras which is isomorphic to $\m m_{n+2}$. Note that $\m M_1$ and $\m M_2$ can easily be represented by letting $\vr M$ be $\{1, \id, \eta_0 \}$ and $\{1, \id, \eta_0, \eta_1 \}$, respectively.  Thus we have

\begin{thm}
For any positive integer $n$, there is a lattice $\vr M$ of equivalence relations on a finite set so that  $\vr M \cong \m M_n$ and $\eqr(\ra(\vr M))= \eqr(\fo(\vr M)) = \vr M$.
\end{thm}

\section{Proof of Lemma \ref{lemma1}} \label{mainsection}

In this section, we prove that  $\eqr(\ra(\vr m)) = \vr m$.  Again, the results in this section are implicit in \cite{Lyndon}, but the relationship is not immediately apparent; hence we provide ``bottom-up" proofs here.

\begin{lm}\label{lemma}
  For distinct atoms $\alpha$ and $\beta$ of $\vr m$, $\alpha\circ\beta=1$.
\end{lm}

\begin{proof}
  There are two nontrivial cases.

  Case 1: $\eta_i\circ\alpha_k$.\\

  Suppose for simplicity that $i=0$.  Then let $\langle u_0,u_1\rangle$ and $\langle v_0,v_1\rangle$ be any pairs in $\mathbb{Z}_p^2$.  We will show $\langle u_0,u_1\rangle\eta_0\circ\alpha_k\langle v_0,v_1\rangle$.  Let $\langle y_0,y_1\rangle=\langle u_0,ku_0+v_1-kv_0\rangle$.  Then $\langle u_0,u_1\rangle\eta_0\langle y_0,y_1\rangle$.  We need to show $\langle y_0,y_1\rangle\alpha_k\langle v_0,v_1\rangle$.  We have
  \begin{align*}
     ky_0-y_1 &=ky_0-(ku_0+v_1-kv_0)\\
     &=ku_0-ku_0-v_1+kv_0\\
     &=kv_0-v_1.\\
  \end{align*}

  Hence $\langle y_0,y_1\rangle\alpha_k\langle v_0,v_1\rangle$.\\

  Case 2: $\alpha_i\circ\alpha_j, i\neq j$.\\

  Again, let $\langle u_0,u_1\rangle,\langle v_0,v_1\rangle \in\mathbb{Z}_p^2$.  We need $\langle y_0,y_1\rangle$ such that $$iu_0-u_1=iy_0-y_1 \text{ and }jy_0-y_1=jv_0-v_1.$$  Since $\mathbb{Z}_p$ is a field, we can find $y_0\in\mathbb{Z}_p$ such that $$(j-i)y_0=u_1-iu_0+jv_0-v_1,$$ so $$iy_0=jy_0-u_1+iu_0-jv_0+v_1,$$

  \noindent and let $$y_1=j(y_0-v_0)+v.$$  Then
  \begin{align*}
    iy_0-y_1 &=jy_0-u_1+iu_0-jv_0+v_1-y_1\\
    & =jy_0-u_1+iu_0-jv_0+v_1-[jy_0-jv_0+v_1]\\
    & =iu_0-u_1.\\
  \end{align*}
  Hence $\langle u_0,u_1\rangle\alpha_i\langle y_0-y_1\rangle$.  Also
  \begin{align*}
    jy_0-y_1 &=(j-i)y_0+iy_0-y_1\\
    &= (j-i)y_0+iy_0-(jy_0-jv_0+v_1)\\
    &=(j-i)y_0+(i-j)y_0+jv_0-v_1\\
    &=jv_0-v_1.\\
  \end{align*}

  Hence $\langle y_0,y_1\rangle\alpha_j\langle v_0,v_1\rangle$.

\end{proof}

Thus we see that $\con(\mathbb{Z}^2_p)\cong \m M_{p+1}$.  Now we
define $\mathcal{M}$ to be the sublattice of Con$(\mathbb{Z}_p^2)$ consisting of the identity
and universal relations, along with the atoms $\eta_0,\eta_1$, and
$\alpha_1$ through $\alpha_{n}$.  Then $\mathcal{M}\cong\m M_{n+2}$.

Now we are ready to prove Lemma \ref{lemma1}.

\begin{proof}[Proof  of Lemma \ref{lemma1}]
  First we establish the following claim: BA$(\mathcal{M})=$RA$(\mathcal{M})$,  where BA$(\mathcal{M})$ is the Boolean algebra generated by $\mathcal{M}$.  The set $At(\text{BA}(\mathcal{M}))$ of atoms of BA$(\mathcal{M})$ consists of the identity relation $\1$ along with $\eta_0\cap \overline{\1},\eta_1 \cap \overline{\1}$, and $\alpha_1\cap \overline{\1}$ through $\alpha_{n}\cap \overline{\1}$, and the one additional atom $$\beta=\overline{\1+\eta_0+\eta_1+\ds\sum^{n}_{i=1}\alpha_i}.$$

  To see this, consider the $\eta _i$'s and $\alpha_j$'s.  Any distinct pair of these intersect to $\1$, so the $\eta_i\cap\overline{\1}$'s and $\alpha_j\cap\overline{\1}$'s are minimal nonzero elements, hence are atoms.  The Boolean algebra generated by the $\eta_i$'s and the $\alpha_j$'s will be the same as that generated by $\1$, the $\eta_i\cap\overline{\1}$'s and the $\alpha_j\cap\overline{\1}$'s.

  By Proposition 4.4 of \cite{SK}, the Boolean algebra generated by these atoms is equal to all joins of meets of atoms and their complements.  Since the meet of an atom $a$ with anything is either $a$ or 0, we need only consider joins of meets of complements of atoms.  Since we are looking for the atoms of the Boolean algebra, we need only consider meets of complements of atoms.  All such meets will be above the meet of all such complements, which is
  $$\beta=\overline{\1}\cdot\overline{n_0}\cdot\overline{n_1}\cdot \prod_{i=1}^{n}\overline{\alpha_i}=\overline{\1+\eta_0+\eta_1+\ds\sum^{n}_{i=1}\alpha_i}$$
  Thus $\beta$ is the only atom not previously listed.

  We need to show that any composition of atoms is already in BA$(\mathcal{M})$.  If $a\in At$(BA$(\mathcal{M})$), then $a\circ a=\1+a$, since $a$ is an equivalence relation,  ``minus the identity", that has no singleton equivalence classes.  If $a\neq b\in At(\text{BA}(\mathcal{M}))$, then $$a\circ
  b=\overline{\1+a+b}.$$

   To establish this, we first prove that $$(\alpha_i\cap \overline{\1})\circ(\alpha_j\cap\overline{\1})=\overline{\1+\alpha_i+\alpha_j}.  $$

To establish the inclusion $(\alpha_i\cap\overline{\1})\circ(\alpha_j\cap\overline{\1})\supseteq\overline{\1+\alpha_i+\alpha_j}$, consider the proof of Case 2 of Lemma \ref{lemma}.  It is not hard to check that if $\langle u_0,u_1\rangle$ and $\langle v_0,v_1\rangle$ are not related by $\1$, $\alpha_i$, or $\alpha_j$, then the pair $\langle y_0,y_1\rangle$ given in the proof is distinct from both $\langle u_0,u_1\rangle$ and $\langle v_0,v_1\rangle$.  Hence if
$$\langle u_0,u_1\rangle\overline{\1 +\alpha_i+\alpha_j}\langle v_0,v_1\rangle,$$
then $$\langle u_0,u_1\rangle(\alpha_i\cap\overline{\1}\circ(\alpha_j\cap\overline{\1})\langle v_0,v_1\rangle.$$

     It remains to show that $(\alpha_i\cap\overline{\1})\circ(\alpha_j\cap\overline{\1})$ contains nothing but $\overline{\1+\alpha_i+\alpha_j}$.  Since $\alpha_i\cap\overline{\1}$ and $\alpha_j\cap\overline{\1}$ are  disjoint symmetric diversity relations, their composition is disjoint from the identity.  To prove that this composition is disjoint from $\alpha_i$ (and by symmetry from $\alpha_j$ as well), suppose for contradiction that there exist pairwise distinct pairs $\langle u_0,u_1\rangle,\langle y_0,y_1\rangle,\langle v_0,v_1\rangle$ with $\langle u_0,u_1\rangle\alpha_i\langle y_0,y_1\rangle\alpha_j\langle v_0,v_1\rangle$ and $\langle u_0,u_1\rangle\alpha_i\langle v_0,v_1\rangle$.  Then $iu_0-u_1=iy_0-y_1=iv_0-v_1$ and $jy_0-y_1=jv_0-v_1$.  Then $(j-i)y_0=(j-i)v_0$, and so $y_0=v_0$.  Then $y_1=v_1$ as well, a contradiction.  By a similar argument,
  $(\eta_i\cap\overline{\1})\circ(\alpha_j\cap\overline{\1})=\overline{\1+\eta_i+\alpha_j}$.


 Therefore, the composition of atoms is a boolean combination of those atoms.  Since composition distributes over union, it follows that BA$(\mathcal{M})=\text{RA}(\mathcal{M})$.\\

  Now we are ready to show that Eq(RA$(\mathcal{M})$)$=\mathcal{M}$.  First, we note that $\eta_0,\eta_1,$ and $\alpha_1$ through $\alpha_{n}$ are the minimal non-identity elements in Eq$(\text{RA}(\mathcal{M}))$.  So suppose that there is some other relation $\gamma\in \text{Eq}(\text{RA}(\mathcal{M}))$.  Since  $\gamma\in$RA$(\mathcal{M})=$BA$(\mathcal{M})$ and since $\gamma$ is not an atom, $\gamma$ must contain  at least two atoms. Since $\gamma$ also contains $\1$, $\gamma$ contains two of the non-trivial relations $\{\eta_0,\eta_1,\alpha_1,\ldots,\alpha_{n}\}$, hence $\gamma\circ \gamma=1$.  But $\gamma\circ \gamma=\gamma$, hence $\gamma=1$.  So the only non-trivial relations in Eq(RA$(\mathcal{M})$ are $\{\eta_0,\eta_1,\alpha_1,\ldots,\alpha_{n}\}$ along with the unit $1$. Therefore Eq$(\text{RA}(\mathcal{M}))=\mathcal{M}$.
\end{proof}


\end{document}